\documentclass{amsart}
\usepackage{delimset, xcolor, pgf, array, tikz}
\usepackage[colorlinks, citecolor=cyan, pagebackref, linkcolor=magenta]{hyperref}
\usepackage{enumitem}
\setlist{itemsep=4pt, topsep=0pt, leftmargin=17pt, listparindent=11pt}

\usepackage{mathtools}
\DeclarePairedDelimiter\floor{\lfloor}{\rfloor}

\usepackage[capitalize]{cleveref}
\crefname{ineq}{Ineq.}{Ineqs.}
\creflabelformat{ineq}{~\upshape(#2#1#3)}
\crefname{ineqs}{Ineqs.}{Ineqs.}
\creflabelformat{ineqs}{~\upshape(#2#1#3)}

\usepackage{booktabs}
\newcolumntype{L}{>{$}l<{$}} 
\newcolumntype{C}{>{$}c<{$}} 

\newtheorem{theorem}{Theorem}[section]
\newtheorem{lemma}[theorem]{Lemma}
\newtheorem{proposition}[theorem]{Proposition}

\numberwithin{equation}{section}

\allowdisplaybreaks

\title[A unimodal sequence with mode at a quarter length]
{A unimodal sequence with mode at a quarter length}

\author[M.Y.C.~Liu]{Max Y.C. Liu}
\address[Max Y.C. Liu]{School of Mathematics and Statistics, Beijing Institute of Technology, Beijing 102400, P. R. China.}
\email{max@bit.edu.cn}

\author[D.G.L.~Wang]{David G.L. Wang$^*$}
\address[David G.L. Wang]{School of Mathematics and Statistics \& MIIT Key Laboratory of Mathematical Theory and Computation in Information Security, Beijing Institute of Technology, Beijing 102400, P. R. China.}
\email{glw@bit.edu.cn}
\thanks{$^*$Wang is the corresponding author, and is supported by the NSFC (Grant No.\ 12171034).}


\begin{document}

\begin{abstract}
We show that the number $A(n,m)$ 
of partitions with $m$ even parts and largest hook length $n$ 
is strongly unimodal with mode $\floor{(n-1)/4}$ for $n\ge 6$.
We establish this result by induction,
using a $5$-term recurrence due to Lin, Xiong and Yan,
and two $4$-term recurrences obtained by Zeilberger's algorithm.
The sequence $A(n,m)$ is not log-concave.
Using M\"obius transformation and the method of interlacing zeros,
we obtain that 
every zero of every generating function $\sum_m A(n,m)z^m$ lies 
on the left half part of the circle $\abs{z-1}=2$. Moreover,
as a direct application of Wang and Zhang's characterization of
root geometry of polynomial sequences that satisfy a recurrence of type $(1,1)$,
we see that all these zeros are densely distributed on the half circle.
\end{abstract}
\maketitle

\section{Introduction}

A sequence $(a_0,\dots,a_n)$ is \emph{unimodal} if there is an index $0\le k\le n$ such that
\[
a_0
\le a_1
\le\dots
\le a_k
\ge a_{k+1}
\ge\dots
\ge a_n.
\]
The study of sequence unimodality 
is a classical subject dating back to Newton.
Unimodality results have a number of applications, 
and are also important in their own right. 
We refer to Brenti~\cite{Bre89,Bre94} and Stanley \cite{Sta89}
as bible references of unimodal sequences arising from combinatorics.
Proving unimodalities is often difficult.
As for they are essentially inequalities,
the methods involve constructing injections, 
inductions, 
considering the stronger properties of log-concavity and real-rootedness,
and so on.

A powerful tool of studying sequences is the WZ-theory.
In particular, Zeilberger's algorithm automatically produces recurrences 
for sequences subject to some technical conditions,
See Petkovsek, Wilf, and Zeilberger \cite{PWZ96B} for a comprehensive reference of the WZ theory.

In this paper, 
we confirm a conjectured unimodality of a combinatorial sequence by induction 
and with the aid of two $4$-term recurrences (\cref{lem:rec:dd,lem:rec:G}) obtained by Zeilberger's algorithm,
and investigate the root geometry of their generating functions.

Let $\mathcal H_n$ be the set of partitions with largest hook length $n$.
Lin, Xiong and Yan \cite{LXY22X} showed
that the number $A(n,m)$ of partitions in $\mathcal H_n$ with exactly $m$ even parts
is also the number of partitions $\lambda=\lambda_1\lambda_2\dotsb\in\mathcal H_n$ 
with $m$ pairs $(\lambda_i,\lambda_{i+1})$ of consecutive parts of the same value $\lambda_i=\lambda_{i+1}$,
and 
\begin{equation}\label{rec:A.5}
A(n,m)
=A(n-1,\,m)
+A(n-1,\,m-1)
+A(n-2,\,m)
-A(n-2,\,m-1),
\end{equation}
for $n\ge 2$, with 
\[
A(0,0)=0
\quad\text{and}\quad
A(1,0)=1.
\] 
When $m=0$, \cref{rec:A.5} reduces to
\[
A(n,0)=A(n-1,\,0)+A(n-2,\,0).
\]
Thus $A(n,0)$ is the $n$th Fibonacci number.
In terms of the generating function
\[
F_n(z)=\sum_{m}A(n,m)z^m,
\]
\cref{rec:A.5} is equivalent to say that
\begin{equation}\label{gf:A}
F_n(z)=(1+z)F_{n-1}(z)+(1-z)F_{n-2}(z)\quad\text{for $n\ge 3$},
\end{equation}
with initial values $F_1(z)=1$ and $F_2(z)=1+z$.
By \cref{rec:A.5}, it is routine to verify that 
\begin{equation}\label{def:A}
A(n,m)=\sum_{k}\binom{k}{m}\binom{n-1-k}{k-m},\quad\text{for all $n,m\in\mathbb Z$.}
\end{equation}
Lin~\cite{Lin22W} conjectures that the sequence $\{A(n,m)\}_m$ is unimodal with mode $m^*=\floor{(n-1)/4}$, i.e.,
\[
A(n,0)
\le A(n,1)
\le\dots
\le A(n,m^*)
\ge A(n,\,m^*+1)
\ge A(n,\,m^*+2)
\ge \dots
\ge A(n,\,n-1).
\]
The values of $A(n,m)$ for small $n$ and $m$ can be found from \cref{tab:A}.
\begin{table}[h]
\centering
\caption{The values of $A(n,m)$ for $1\le n\le 15$ and $0\le m\le n-1$.}
\label{tab:A}
\begin{tabular}{LLLLLLLLLLLLLLLL}
\toprule
&\multicolumn{15}{C}{m}\\
\cmidrule(r){2-16}
n & 0 & 1 & 2 & 3 & 4 & 5 & 6 & 7 & 8 & 9 & 10 & 11 & 12 & 13 & 14\\ 
\midrule
1
& 1\\
2
& 1
& 1\\
3
& 2
& 1
& 1\\
4
& 3
& 3
& 1
& 1\\
5
& 5
& 5
& 4
& 1
& 1
\\
6
& 8
& 10
& 7
& 5
& 1
& 1\\
7
& 13
& 18
& 16
& 9
& 6
& 1
& 1\\ 
8
& 21
& 33
& 31
& 23
& 11
& 7
& 1
& 1\\ 
9
& 34
& 59
& 62
& 47
& 31
& 13
& 8
& 1
& 1\\
10
& 55
& 105
& 119
& 101
& 66
& 40
& 15
& 9
& 1
& 1\\ 
11
& 89
& 185
& 227
& 205
& 151
& 88
& 50
& 17
& 10
& 1
& 1\\ 
12
& 144
& 324
& 426
& 414
& 321
& 213
& 113
& 61
& 19
& 11
& 1
& 1\\
13
& 233
& 564
& 792
& 818
& 681
& 471
& 288
& 141
& 73
& 21
& 12
& 1
& 1\\ 
14
& 377
& 977
& 1458
& 1598
& 1406
& 1044
& 659
& 377
& 172
& 86
& 23
& 13
& 1
& 1\\ 
15
& 610
& 1685
& 2663
& 3082
& 2867
& 2240
& 1520
& 889
& 481
& 206
& 100
& 25
& 14
& 1
& 1\\ 
\bottomrule
\end{tabular}
\end{table}
We confirm his conjecture by establishing a strong unimodality.
\begin{theorem}\label{thm:unimodal}
For $n\ge 6$, the sequence $\{A(n,m)\}_{m=0}^{n-2}$ is strongly unimodal with mode $m^*$, i.e.,
\[
A(n,0)
< A(n,1)
<\dots
< A(n,m^*)
> A(n,\,m^*+1)
> A(n,\,m^*+2)
> \dots
> A(n,\,n-2),
\]
where $m^*=\floor{(n-1)/4}$.
\end{theorem}

The sequence $\{A(n,m)\}_m$ for $n\ge 3$ is not log-concave, since
\[
A(n,\,n-2)^2-A(n,\,n-3)A(n,\,n-1)=2-n<0.
\]
More evidence for the non-log-concavity can be seen from \cref{Delta:m+2,Delta:m+3,Delta:m+4,Delta:m+5}.

The other result of this paper is the zero distribution of the generating functions $F_n(z)$.
\begin{theorem}\label{thm:rg}
The zeros of all generating functions $F_n(z)$ are densely distributed
on the left half part of the circle 
\[
\abs{z-1}=2.
\]
\end{theorem}

See~\cref{fig:limit} for illustration.
We will show \cref{thm:unimodal,thm:rg} respectively
in \cref{sec:unimodal,sec:zeros}.
\usetikzlibrary {arrows.meta}
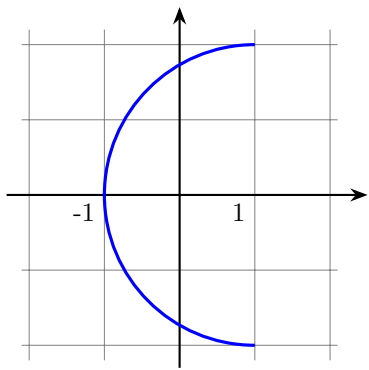
\begin{figure}[h]
\begin{tikzpicture}[>=Stealth]
\draw[step=1cm, gray, very thin] (-0.1,-2.2) grid (4.1, 2.2); 
\draw[thick, ->] (-0.3, 0) -- (4.5, 0);
\draw[thick, ->] (2, -2.3) -- (2, 2.5);
\draw (1, 0) node[anchor=north east] {-1};
\draw (3, 0) node[anchor=north east] {1};
\draw[very thick, color=blue] (3,2) arc [start angle=90, end angle=270, radius=2];
\end{tikzpicture}
\caption{The zeros of all generating functions $F_n(z)$ are densely distributed on the left half part of the circle $|z-1|=2$.}
\label{fig:limit}
\end{figure}

\section{The quarter unimodality --- A proof for \cref{thm:unimodal}}\label{sec:unimodal}

Let $n\ge 6$ and 
\[
\Delta(n,m)=A(n,m)-A(n,m-1).
\]
The values of $\Delta(n,m)$ for small $n$ and $m$ can be found in \cref{tab:dd}.
\begin{table}[h]
\centering
\caption{The values of $\Delta(n,m)$ for $0\le m\le 14$ and $1\le n\le m+1$.}
\label{tab:dd}
\begin{tabular}{LLLLLLLLLLLLLLLL}
\toprule
&\multicolumn{15}{C}{m}\\
\cmidrule(r){2-16}
n & 0 & 1 & 2 & 3 & 4 & 5 & 6 & 7 & 8 & 9 & 10 & 11 & 12 & 13 & 14\\ 
\midrule
1
& 1\\
2
& 1
& 0\\
3
& 2
& -1
& 0\\
4
& 3
& 0
& -2
& 0\\
5
& 5
& 0
& -1
& -3
& 0
\\
6
& 8
& 2
& -3
& -2
& -4
& 0\\
7
& 13
& 5
& -2
& -7
& -3
& -5
& 0\\ 
8
& 21
& 12
& -2
& -8
& -12
& -4
& -6
& 0\\ 
9
& 34
& 25
& 3
& -15
& -16
& -18
& -5
& -7
& 0\\
10
& 55
& 50
& 14
& -18
& -35
& -26
& -25
& -6
& -8
& 0\\ 
11
& 89
& 96
& 42
& -22
& -54
& -63
& -38
& -33
& -7
& -9
& 0\\ 
12
& 144
& 180
& 102
& -12
& -93
& -108
& -100
& -52
& -42
& -8
& -10
& 0\\
13
& 233
& 331
& 228
& 26
& -137
& -210
& -183
& -147
& -68
& -52
& -9
& -11
& 0\\ 
14
& 377
& 600
& 481
& 140
& -192
& -362
& -385
& -282
& -205
& -86
& -63
& -10
& -12
& 0\\ 
15
& 610
& 1075
& 978
& 419
& -215
& -627
& -720
& -631
& -408
& -275
& -106
& -75
& -11
& -13
& 0\\ 
\bottomrule
\end{tabular}
\end{table}
Using \cref{def:A},
one may calculate by bootstrap that for $m\ge 1$,
\begin{align}
\label{Delta:m+1}
\Delta(m+1,\,m)&=0,\\
\label{Delta:m+2}
\Delta(m+2,\,m)&=-m,\\
\label{Delta:m+3}
\Delta(m+3,\,m)&=-(m-1),\\
\label{Delta:m+4}
\Delta(m+4,\,m)&=-(m+4)(m-1)/2,\\
\label{Delta:m+5}
\Delta(m+5,\,m)&=-(m^2+m-4).
\end{align}
Since 
\[
m\le m^*
\iff	
n\ge 4m+1,
\]
\cref{thm:unimodal} can be restated as
\[
\begin{cases}
\Delta(n,m)<0,&\text{if $m+2\le n\le 4m$,}\\ 
\Delta(n,m)>0,&\text{if $n\ge 4m+1$}.
\end{cases}
\]
We will handle the cases $n\le 4m$, $n\in\{4m+1,\,4m+2\}$,
and $n\ge 4m+3$, respectively in \cref{prop:n<=4m,prop:4m+1:4m+2,prop:4m+3}.

A key tool to establish \cref{prop:n<=4m,prop:4m+1:4m+2} is
a recurrence for the following sequence $\{\Delta(n,m)\}_n$, 
which is found by using Zeilberger's algorithm straightforwardly.

\begin{lemma}\label{lem:rec:dd}
For any integers $n$ and $m$,
\[
\sum_{j=0}^3 d_j(n,m)\Delta(n-j,m)=0,
\]
where
\begin{align*}
d_0(n,m)&=-(n-m)(n-4m-1),\\
d_1(n,m)&=2(n-m-1)(n-4m),\\
d_2(n,m)&=2m(n-4m-1),\\
d_3(n,m)&=-(n-1)(n-4m).
\end{align*}
\end{lemma}
\begin{proof}
By \cref{def:A},
\[
\Delta(n,m)
=\sum_k\brk[s]4{
\binom{k}{m}\binom{n-1-k}{k-m}
-\binom{k}{m-1}\binom{n-1-k}{k-m+1}}.
\]
It is routine to verify that it satisfies the recurrence.
This completes the proof.
\end{proof}

We remark that the recurrence in \cref{lem:rec:dd} 
with any finite number of initial values of $A(n,m)$
is not enough to generate the whole set of numbers $\Delta(n,m)$. This is because
the recurrence for $n=4m+1$ 
degenerates to the identity 
\begin{equation}\label{dd:4m}
3\Delta(4m,m)=2\Delta(4m-2,\,m),
\end{equation}
which does not provide any information to determine the value of $\Delta(4m+1,\,m)$.

\begin{proposition}\label{prop:n<=4m}
Let $n\ge 6$. 
Then $\Delta(n,m)<0$ for $m+2\le n\le 4m$.
\end{proposition}
\begin{proof}
For convenience, we rewrite \cref{lem:rec:dd} as
\begin{equation}\label{XY}
(n-1-4m)X(n,m)=(n-4m)Y(n-1,\,m),
\end{equation}
where 
\begin{align*}
X(n,m)&=-(n-m)\Delta(n,m)+2m\Delta(n-2,\,m)\quad\text{and}\\
Y(n,m)&=-2(n-m)\Delta(n,m)+n\Delta(n-2,\,m).
\end{align*} 
We shall show the following inequalities
\begin{align}
\label[ineq]{dsr:X}
X(n,m)&>0,
\quad\text{if $m+3\le n\le 4m-1$,}\\
\label[ineq]{dsr:Y}
Y(n,m)&>0,
\quad\text{if $m+3\le n\le 4m-1$, and}\\
\label[ineq]{dd<0}
\Delta(n,m)
&<0,\quad\text{if $m+2\le n\le 4m$.}
\end{align}
By \cref{Delta:m+1,Delta:m+3}, one may calculate that
\[
X(m+3,\,m)=3(m-1)>0
\quad\text{and}\quad
Y(m+3,\,m)=6(m-1)>0.
\]
From \cref{Delta:m+2,Delta:m+3}, we see that
\[
\Delta(m+2,\,m)=-m<0
\quad\text{and}\quad
\Delta(m+3,\,m)=1-m<0.
\]
Below we can suppose that $n\ge m+4$.
In view of \cref{dd:4m}, 
the sign of $\Delta(4m,m)$ is same to that of $\Delta(4m-2,\,m)$,
and we can restrict ourselves to $n\le 4m-1$.

Let $m+4\le n\le 4m-1$. We proceed by induction on $n$.
By the induction hypothesis of \cref{dsr:Y},
\[
Y(n-1,\,m)>0,
\]
which implies \cref{dsr:X} by \cref{XY}.
By the induction hypothesis of \cref{dd<0},
\begin{equation}\label[ineq]{pf:dd.n-2<0}
\Delta(n-2,\,m)<0.
\end{equation}
Therefore, by \cref{dsr:X}, we can infer that
\begin{equation}\label[ineq]{dd.n/dd.n-2}
\frac{\Delta(n,m)}{\Delta(n-2,\,m)}
>\frac{2m}{n-m}
>\frac{n}{2(n-m)},
\end{equation}
which can be recast as \cref{dsr:Y}.
Since the rightmost side of \cref{dd.n/dd.n-2} is positive, 
by \cref{pf:dd.n-2<0}, we obtain \cref{dd<0}.
This completes the proof.
\end{proof}

A key to establish \cref{prop:4m+1:4m+2} is \cref{lem:7/9},
in which the lower bound $1/m$ is technically set for
making the induction proof for the upper bound $7/9$ valid.

\begin{lemma}\label{lem:7/9}
Let $m\ge 5$. Then
\[
\frac{1}{m}<\frac{\Delta(n-1,\,m)}{\Delta(n,m)}<\frac{7}{9},
\quad\text{if $m+4\le n\le 4m-2$}.
\]
\end{lemma}
\begin{proof}
For $5\le m\le 9$, it is routine to check its truth with the aid of mathematical software.
Let 
\[
m\ge 10,\quad
\alpha=1/m,\quad\text{and}\quad
\beta=7/9.
\] 
The bound $m\ge 10$ will be used to show 
$s_0(4m-2,\,m)<0$ at the end of this proof.
By \cref{Delta:m+3,Delta:m+4,Delta:m+5}, 
it is routine to verify that 
\begin{align*}
\frac{\Delta(m+3,\,m)}{\Delta(m+4,\,m)}
&=\frac{-(m-1)}{-(m+4)(m-1)/2}
\in(\alpha,\beta),\quad\text{and}\\
\frac{\Delta(m+4,\,m)}{\Delta(m+5,\,m)}
&=\frac{-(m+4)(m-1)/2}{-(m^2+m-4)}
\in(\alpha,\beta).
\end{align*}
Proceeding by induction on $n$,
we suppose that $n\le 4m-1$. Then
\begin{align}
\label[ineq]{pf:sign:d1}
d_1(n,m)&=2(n-m-1)(n-4m)<0,\quad\text{and}\\
\label[ineq]{pf:sign:d3}
d_3(n,m)&=-(n-1)(n-4m)>0.
\end{align}
By \cref{prop:n<=4m},
\begin{align}
\label[ineq]{pf:dd.n-2<0}
\Delta(n-2,\,m)&<0,\\
\label[ineq]{pf:dd.n-1<0}
\Delta(n-1,\,m)&<0,
\quad\text{and}\\
\label[ineq]{pf:dd.n<0}
\Delta(n,m)&<0.
\end{align}
By \cref{pf:dd.n-2<0,pf:dd.n-1<0} and by induction hypothesis for $n-1$ and $n-2$, we obtain
\begin{align}
\label[ineq]{pf:n-2:n-3}
\beta \Delta(n-2,\,m)
&<\Delta(n-3,\,m)
<\alpha\Delta(n-2,\,m)
\quad\text{and}\\
\label[ineq]{pf:n-1:n-2}
\beta \Delta(n-1,\,m)
&<\Delta(n-2,\,m)
<\alpha\Delta(n-1,\,m).
\end{align}
By \cref{pf:sign:d3,pf:n-2:n-3}, we deduce that
\begin{equation}\label[ineq]{pf:s2:t2}
s_2(n,m)\Delta(n-2,\,m)
<\sum_{j=2}^3 d_j(n,m)\Delta(n-j,\,m)
<t_2(n,m)\Delta(n-2,\,m),
\end{equation}
where
\begin{align}
\label{def:s2}
s_2(n,m)
&=d_2(n,m)+\beta d_3(n,m)
=-\beta n^2
+(4\beta m+2m+\beta)n
-2m(4m+2\beta+1),
\quad\text{and}\\
\notag
t_2(n,m)
&=d_2(n,m)+\alpha d_3(n,m)
=-\alpha n^2
+(4\alpha m+2m+\alpha)n
-2m(4m+2\alpha+1).
\end{align}

First, we show the lower bound $\alpha$.
We compute that 
\[
mt_2(n,m)
=-n^2
+(2m^2+4m+1)n
-2m(4m^2+m+2).
\]
Since 
\[
n<4m<-\frac{2m^2+4m+1}{2(-1)}
\quad\text{and}\quad
mt_2(4m,m)
=-2m^2<0,
\]
we find
\[
t_2(n,m)<0,\quad\text{for all $n\le 4m$.}
\]
Then by \cref{lem:rec:dd,pf:s2:t2,pf:n-2:n-3},
\begin{align}
\notag
-md_0(n,m)\Delta(n,m)
&=md_1(n,m)\Delta(n-1,\,m)+m\sum_{j=2}^3 d_j(n,m)\Delta(n-j,\,m)\\
\notag
&<md_1(n,m)\Delta(n-1,\,m)+mt_2(n,m)\Delta(n-2,\,m)\\
\label[ineq]{pf:d0:f1:3}
&<t_1(n,m)\Delta(n-1,\,m),
\end{align}
where
\begin{align*}
t_1(n,m)
&=md_1(n,m)+\beta mt_2(n,m)\\
&=(2m-\beta)n^2
+(-10m^2+2\beta m^2+4\beta m-2m+\beta)n\\
&\quad-2m(-4m^2+4\beta m^2-4m+\beta m+2\beta)
\end{align*}
is a quadratic polynomial in $n$ with leading coefficient $2m-\beta>0$.
Since $m\ge 10$ and $\beta>0$, we have
\[
t_1(m+1,\,m)
=-\beta m(6m^2-3m+1)
<0
\quad\text{and}\quad
t_1(4m,\,m)
=-2\beta m^2<0.
\]
Therefore,
\[
t_1(n,m)<0,\quad\text{for all $m+1\le n\le 4m$}.
\]
Together with \cref{pf:dd.n<0}, we can recast \cref{pf:d0:f1:3} as
\[
\frac{\Delta(n-1,\,m)}{\Delta(n,m)}
>\frac{md_0(n,m)}{-t_1(n,m)}.
\]
Hence for the lower bound $\alpha$,
it suffices to show that 
\[
m d_0(n,m)+\alpha t_1(n,m)>0.
\]
Since $\alpha=1/m$, it suffices to show that 
\[
t_0(n,m)>0,\quad\text{for all $m+4\le n\le 4m-2$},
\]
where
\begin{align*}
t_0(n,m)
&=m^2 d_0(n,m)+\alpha m t_1(n,m)\\
&=(-m^2+2m-\beta)n^2
+\brk[s]1{5m^3+(2\beta-9)m^2+(4\beta-2)m+\beta}n\\
&\quad-m\brk[s]1{4m^3+(8\beta-7)m^2+(2\beta-8)m+4\beta}
\end{align*}
is a quadratic polynomial in $n$ with leading coefficient 
\[
-m^2+2m-\beta=-(m-1)^2+1-\beta<0,\quad\text{since $m\ge 10$ and $\beta>0$.}
\]
Again, we compute 
\begin{align*}
t_0(m+4,\,m)
&=6(2-\beta)m^3
+3(3\beta-10)m^2
+(24+5\beta)m
-12\beta,\quad\text{and}\\
t_0(4m,\,m)
&=m^2(3m-2\beta).
\end{align*}
Since $m\ge 10$ and $\beta\in(0,1)$, 
it is elementary to show that these two numbers are positive.
Therefore,
\[
t_0(n,m)>0,
\quad\text{for all $m+4\le n\le 4m$}.
\]
This proves the lower bound $\alpha$.

Second, we show the desired upper bound $\beta$.
By \cref{lem:rec:dd,pf:s2:t2},
\[
-d_0(n,m)\Delta(n,m)
>d_1(n,m)\Delta(n-1,\,m)+s_2(n,m)\Delta(n-2,\,m).
\]
By \cref{pf:n-1:n-2}, we obtain
\begin{equation}\label[ineq]{pf:s1}
-d_0(n,m)\Delta(n,m)
>s_1(n,m)\Delta(n-1,\,m),
\end{equation}
where
\begin{equation}\label{def:s1}
s_1(n,m)=\begin{cases}
d_1(n,m)+\beta s_2(n,m),&\text{if $s_2(n,m)>0$},\\
d_1(n,m)+\alpha s_2(n,m),&\text{if $s_2(n,m)\le 0$}.
\end{cases}
\end{equation}

We claim that 
\[
s_1(n,m)<0,
\quad\text{for all $m+4\le n\le 4m-2$}.
\]
In view of \cref{pf:sign:d1,def:s1}, we obtain the claim 
as if $s_2(n,m)\le 0$.
When $s_2(n,m)>0$, 
\[
s_1(n,m)
=(2-\beta^2)n^2+(-10m+2\beta m+4\beta^2 m-2+\beta^2)n+2m(4m-4\beta m+4-\beta-2\beta^2)
\]
is a quadratic polynomial with leading coefficient $2-\beta^2>0$.
Since $m\ge 10$ and $\beta\in(0,1)$, we find
\[
s_1(m+1,\,m)
=\beta m\brk[s]1{3(\beta-2)m-\beta}
<0
\quad\text{and}\quad
s_1(4m,m)
=-2\beta m<0.
\]
Therefore, we obtain the claimed inequality.

Now, \cref{pf:s1} can be recast as
\begin{equation}\label[ineq]{pf:ratio:s1}
\frac{\Delta(n-1,\,m)}{\Delta(n,m)}
<\frac{d_0(n,m)}{-s_1(n,m)}.
\end{equation}
In order to show the upper bound $\beta$,
it suffices to show that $s_0(n,m)<0$, where
\[
s_0(n,m)=d_0(n,m)+\beta s_1(n,m).
\]
We deal with the two cases separately.

\noindent Case I. $s_2(n,m)\ge 0$.
In this case, 
\begin{align*}
s_0(n,m)
&=d_0(n,m)+\beta \brk1{d_1(n,m)+\beta s_2(n,m)}\\
&=(1-\beta)(\beta^2+\beta-1)n^2
+\brk[s]1{(4\beta^3+2\beta^2-10\beta+5)m+(\beta - 1)(\beta^2 + \beta - 1)}n\\
&\quad -m\brk[s]1{4(2\beta^2-2\beta+1)m+(4\beta^3+2\beta^2-8\beta+1)}.
\end{align*}
Since $s_2(n,m)\ge 0$, we obtain from \cref{def:s2} that
\[
n^2
\le \frac{(4\beta m+2m+\beta)n
-2m(4m+2\beta+1)}{\beta}.
\]
Since $\beta=7/9$, we find $(1-\beta)(\beta^2+\beta-1)>0$. Since $n\le 4m-2$, we can infer that
\begin{align*}
s_0(n,m)
&\le(1-\beta)(\beta^2+\beta-1)\frac{(4\beta m+2m+\beta)n-2m(4m+2\beta+1)}{\beta}\\
&\quad +\brk[s]1{(4\beta^3+2\beta^2-10\beta+5)m+(\beta - 1)(\beta^2 + \beta - 1)}n\\
&\quad -m\brk[s]1{4(2\beta^2-2\beta+1)m+(4\beta^3+2\beta^2-8\beta+1)}\\
&=\frac{(\beta-2)m}{\beta}\cdotp \brk[s]1{(1-2\beta)n+8\beta m-4m-1}\\
&\le\frac{(\beta-2)m}{\beta}\cdotp \brk[s]1{(1-2\beta)(4m-2)+8\beta m-4m-1}\\
&=\frac{(4\beta-3)(\beta-2)m}{\beta}
<0.
\end{align*}

\noindent Case II. $s_2(n,m)\le 0$. 
In this case,
\begin{align*}
ms_0(n,m)
&=md_0(n,m)+\beta m\brk1{d_1(n,m)+\alpha s_2(n,m)}\\
&=(2\beta m-m-\beta^2)n^2
-(10\beta m^2-5m^2-4\beta^2 m-m-\beta^2)n\\
&\quad +m(8\beta m^2-4m^2-m-2\beta-4\beta^2)
\end{align*}
is a quadratic polynomial in $n$ with leading coefficient
\[
2\beta m-m-\beta^2
\ge 10(2\beta-1)-\beta^2
>0,\quad\text{since $m\ge 10$ and $\beta=7/9$.}
\] 
Again, we compute
\begin{align*}
ms_0(m+2,m)
&=(3\beta^2-12\beta+6)m^2+(\beta^2+6\beta-2)m-2\beta^2,
\quad\text{and}\\
ms_0(4m-2,\,m)
&=(9-12\beta)m^2+(8\beta^2 + 6\beta - 6)m-6\beta^2.
\end{align*}
Since $m\ge 10$ and $\beta=7/9$,
it is elementary to show that these two numbers are negative.
Therefore,
\[
s_0(n,m)<0\quad\text{for all $m+2\le n\le 4m-2$}.
\]
This proves the desired upper bound $\beta$, and completes the whole proof.
\end{proof}

We remark that the range $m+4\le n\le 4m-2$ is sharp. In fact, when $n=m+3$,
we can compute by \cref{Delta:m+2,Delta:m+3} that
\[
\frac{\Delta(m+2,\,m)}{\Delta(m+3,m)}=\frac{m}{m-1}>1.
\]
For $n=4m-1$, with the aid of mathematical software,
one may see that 
\[
\frac{\Delta(4m-2,\,m)}{\Delta(4m-1,\,m)}>\frac{7}{9},
\quad\text{for all $1\le m\le 100$}.
\]
We still need an upper bound for the ratio $\Delta(4m-2,\,m)/\Delta(4m-1,\,m)$,
which will be used in the proof of \cref{prop:4m+1:4m+2}.

\begin{lemma}\label{lem:10/9}
Let $m\ge 3$. Then
\[
\frac{\Delta(4m-2,\,m)}{\Delta(4m-1,\,m)}
<\frac{10}{9}.
\]
\end{lemma}
\begin{proof}
It is routine to check that
\[
\frac{\Delta(10,\,3)}{\Delta(11,\,3)}
=\frac{9}{11}<\frac{10}{9}
\quad\text{and}\quad
\frac{\Delta(14,\,4)}{\Delta(15,\,4)}
=\frac{192}{215}<\frac{10}{9}.
\]
Let $m\ge 5$. 
We replay the proof for the upper bound in \cref{lem:7/9}. 
Setting $n=4m-1$ in \cref{def:s2}, we obtain
\[
s_2(4m-1,\,m)
=-\frac{8m+14}{9}
<0.
\]
Then setting $n=4m-1$ in \cref{def:s1}, we obtain
\[
s_1(4m-1,\,m)=
d_1(4m-1,\,m)+\frac{s_2(4m-1,\,m)}{m}
=-\frac{2(27m^2-14m+7)}{9m}
<0.
\]
Since
\[
d_0(4m-1,\,m)+\frac{10}{9} s_1(4m-1,\,m)
=-\frac{2(27m^2 - 59m + 70)}{81m}<0,
\]
we deduce from \cref{pf:ratio:s1} that
\[
\frac{\Delta(4m-2,\,m)}{\Delta(4m-1,\,m)}
<\frac{d_0(4m-1,\,m)}{-s_1(4m-1,\,m)}
<\frac{10}{9}.
\]
This completes the proof.
\end{proof}

In order to show the positivity of $\Delta(4m+1,\,m)$, we consider the function
\begin{equation}\label{def:G}
G(n,m)=\Delta(n,m)+\Delta(n-1,\,m).
\end{equation}
With the aid of Zeilberger's algorithm, we find a recurrence for $G(4m+1,\,m)$.
\begin{lemma}\label{lem:rec:G}
For all integers $m$,
\[
G(4m+1,\,m)
=\frac{-(12m-7)G(4m,m)
+(2m-3)G(4m-1,\,m)
+(8m-2)G(4m-2,\,m)}{3m+1}.
\]
\end{lemma}
\begin{proof}
From the definition
\begin{multline*}
G(n,m)
=\sum_k\left[
\binom{k}{m}\binom{n-1-k}{k-m}
-\binom{k}{m-1}\binom{n-1-k}{k-m+1}\right.\\
\quad+\left.
\binom{k}{m}\binom{n-2-k}{k-m}
-\binom{k}{m-1}\binom{n-2-k}{k-m+1}\right],
\end{multline*}
we see that the summand in $G(n,m)$ is hypergeometric.
Invoking Zeilberger's algorithm, we obtain
\[
\sum_{j=0}^3
l_j(n,m)G(n+j,\,m)=0,
\]
for all integers $n$ and $m$,
where
\begin{align*}
l_0(n,m)
&=(n+1)(2n^2-13mn+10n+20m^2-32m+12),\\
l_1(n,m)
&=-m(4n^2-26mn+21n+40m^2-74m+29),\\
l_2(n,m)
&=-4n^3
+30mn^2-24n^2
-66m^2n+119mn-44n
+40m^3-140m^2+109m-24,\quad\text{and}\\
l_3(n,m)
&=(n-m+3)(2n^2-13mn+6n+20m^2-19m+4).
\end{align*}
Setting $n=4m-2$ and rearranging the terms give the desired recurrence.
\end{proof}

Now we are able to
confirm the positivity of $\Delta(4m+1,\,m)$ and $\Delta(4m+2,\,m)$.

\begin{proposition}\label{prop:4m+1:4m+2}
Let $n\ge 6$. Then 
\[
\Delta(n,m)>0
\quad\text{for $n\in\{4m+1,\,4m+2\}$}.
\]
\end{proposition}
\begin{proof}
The truth for $1\le m\le 4$ can be checked directly:
\begin{align*}
&\Delta(6,1)=2,\qquad
\Delta(9,2)=3,\qquad
\Delta(13,3)=26,\qquad
\Delta(17,4)=294,\\
&
\Delta(10,2)=14,\qquad
\Delta(14,3)=140,\qquad
\Delta(18,4)=1541.
\end{align*}
Let $m\ge 5$. For convenience, we write
\[
\Delta(n,m)=\Delta(n)
\quad\text{and}\quad
G(n,m)=G(n).
\]
In view of \cref{lem:rec:G}, we consider the function
\[
(3m+1)G(4m+1)
=-(12m-7)G(4m)
+(2m-3)G(4m-1)
+(8m-2)G(4m-2).
\]
By \cref{dd:4m}, we can express it as 
\begin{align*}
(3m+1)G(4m+1)
&=-(12m-7)\brk1{\Delta(4m)+\Delta(4m-1)}
+(2m-3)\brk1{\Delta(4m-1)+\Delta(4m-2)}\\
&\quad+(8m-2)\brk1{\Delta(4m-2)+\Delta(4m-3)}\\
&=2(4m-1)\Delta(4m-3)
+\brk2{2m-\frac{1}{3}}\Delta(4m-2)
-2(5m-2)\Delta(4m-1).
\end{align*}
By \cref{lem:7/9,prop:n<=4m,lem:10/9},
\[
\Delta(4m-3)>\frac{7}{9}\Delta(4m-2)
\quad\text{and}\quad
\Delta(4m-2)>\frac{10}{9}\Delta(4m-1).
\]
Therefore,
\begin{align*}
(3m+1)G(4m+1)
&>\brk3{\frac{7}{9}\cdotp 2(4m-1)
+\brk2{2m-\frac{1}{3}}}\Delta(4m-2)
-2(5m-2)\Delta(4m-1)\\
&=\frac{74m-17}{9}\cdotp \Delta(4m-2)
-2(5m-2)\Delta(4m-1)\\
&>\brk3{\frac{10}{9}\cdotp\frac{74m-17}{9}-2(5m-2)}\Delta(4m-1)\\
&=\frac{2(77-35m)}{81}\cdotp \Delta(4m-1)
>0.
\end{align*}
By \cref{def:G,prop:n<=4m}, we derive that
\begin{equation}\label[ineq]{ineq:4m+1}
\Delta(4m+1)>-\Delta(4m)>0.
\end{equation}

When $n=4m+2$, \cref{lem:rec:dd} reduces to
\begin{equation}\label{rec:dd.4m+2}
4(3m+1)\Delta(4m+1)
+2m\Delta(4m)
=2(4m+1)\Delta(4m-1)
+(3m+2)\Delta(4m+2).
\end{equation}
By \cref{ineq:4m+1}, the left side of \cref{rec:dd.4m+2} is positive, 
and so is the right side:
\[
2(4m+1)\Delta(4m-1)
+(3m+2)\Delta(4m+2)
>0.
\]
By \cref{prop:n<=4m}, we know that $\Delta(4m-1)<0$.
Thus $\Delta(4m+2)>0$.
\end{proof}

\begin{proposition}\label{prop:4m+3}
Let $m\ge 0$. The sequence $\{\Delta(n,m)\}_{n\ge \max(6,\,4m+1)}$ is increasing.
As a consequence, 
\[
\Delta(n,m)>0\quad\text{if $n\ge 4m+3$}.
\]
\end{proposition}
\begin{proof}
Consider the function 
\[
g(n,m)=\Delta(n,m)-\Delta(n-1,\,m).
\]
We need to show that 
\begin{equation}\label[ineq]{dsr:g>0}
g(n,m)>0,\quad\text{for all $n\ge \max(6,\,4m+1)$.}
\end{equation}
When $m=0$, the number 
\[
g(n,0)
=\Delta(n,0)-\Delta(n-1,\,0)
=A(n,0)-A(n-1,\,0)
=A(n-2,\,0)
\]
is the $(n-2)$th Fibonacci number. 
Since $n\ge 6$, we find $g(n,0)>0$.
We proceed by induction on $m$ and let $m\ge 1$.
Now, we will show \cref{dsr:g>0}, and simultaneously
\begin{equation}\label[ineq]{dsr:dd>0}
\Delta(n,m)>0,\quad\text{for all $n\ge \max(6,\,4m+1)$}.
\end{equation}
Let $n\ge 6$. 
By \cref{prop:4m+1:4m+2}, 
\cref{dsr:dd>0} is true if $n\in\{4m+1,\,4m+2\}$.
Showing \cref{dsr:dd>0} by induction on $n$, we can suppose that
\begin{align}
\label[ineq]{pf:dd.n-2}
\Delta(n-2,\,m)&>0\quad\text{and}\\
\label[ineq]{pf:dd.n-1}
\Delta(n-1,\,m)&>0.
\end{align}
Taking the order-one difference of \cref{rec:A.5} with respect to $m$, 
we obtain
\[
g(n,m)=g(n-1,\,m-1)+\Delta(n-2,\,m)>0
\]
by the induction hypothesis for $m-1$ and \cref{pf:dd.n-2}.
It follows that 
\[
\Delta(n,m)
=g(n,m)+\Delta(n-1,\,m)>0
\]
by \cref{pf:dd.n-1}.
This completes the proof of \cref{dsr:g>0,dsr:dd>0}.
\end{proof}

Combining \cref{prop:n<=4m,prop:4m+1:4m+2,prop:4m+3} together,
we obtain a proof for \cref{thm:unimodal}.

We think the proof is a bit skillful at choosing the bounds $7/9$ and $10/9$
and show. A simple injective proof for \cref{thm:unimodal} will be interesting.

\section{Root geometry of the generating functions --- A proof for \cref{thm:rg}}
\label{sec:zeros}

Let $f$ and $g$ be real-rooted polynomials
with zero sets $\{s_1,\dots,s_n\}_<$ and $\{t_1,\dots,t_{n-1}\}_<$, respectively.
We say that $g$ \emph{strictly interlaces} $f$ if 
\[
s_1<t_1<s_2<t_2<\dots<s_{n-1}<t_{n-1}<s_n.
\]
It is quite common to show real-rootedness of 
polynomials that satisfy certain recurrence by using 
the method of interlacing zeros, see Liu and Wang~\cite{LW07}.

\begin{proposition}\label{prop:zeros}
Let $n\ge 2$.
Then every zero of the generating function $F_n(z)$ lies on the left half part of the circle 
\[
\abs{z-1}=2.
\] 
\end{proposition}
\begin{proof}
We transform the problem of ``circular zeros'' 
to a problem of real-rootedness by considering the function
\[
W_n(z)
=\frac{e^{in\pi/4}(z-i)^{n}}{2^{n/2}}F_{n+1}\brk3{\frac{2(z+i)}{iz+1}+1}.
\]
Here is how this function was born.
\begin{enumerate}
\item
By mapping $z\mapsto 2z+1$, the desired result 
can be restated as that every zero of the polynomial
\[
G_n(z)=F_n(2z+1)
\]
lies on the left half part of the unit circle.
\item
By the M\"obius transformation $z\mapsto(z+i)/(iz+1)$,
the previous desired zero distribution becomes that
the function 
\[
H_n(z)=G_n\brk3{\frac{z+i}{iz+1}}
\]
has only non-positive zeros.
\item
Multiplying $H_n(z)$ by $(z-i)^{n-1}$, 
we can lift the rational functions $H_n(z)$ up to polynomials.
\item
By the translation $n\mapsto n+1$, we can make the entire polynomial sequence 
a bit more standard in the sense that the $n$th polynomial has degree $n$.
\item
Multiplying the polynomials $(z-i)^n H_{n+1}(z)$ by $e^{in\pi/4}$,
we can rotate their complex coefficients to real ones.
\item
Dividing the resulting real polynomials by $2^{n/2}$, 
we can compress their irrational coefficients to rational ones.
\end{enumerate}
It suffices to show that $W_n(z)$ has only negative zeros.
In terms of $W_n(z)$, \cref{gf:A} becomes
\begin{equation}\label{pf:rec:W}
W_n(z)
=2(z+1)W_{n-1}(z)
-(z^2+1)W_{n-2}(z),\quad\text{for $n\ge 2$},
\end{equation}
with initial values
\[
W_0(z)=1
\quad\text{and}\quad
W_1(z)=2(z+1).
\]
By \cref{pf:rec:W}, it is easy to show by induction that 
$W_n(z)$ is a polynomial of degree $n$ with leading coefficient $n+1$.
It is observable that the polynomials
\[
T_n(z)=z^n W_n(1/z)
\]
satisfy the same recurrence \cref{pf:rec:W} and have the same initial values $T_0=W_0$
and $T_1=W_1$. Thus $T_n=W_n$ for all $n$, which implies that $W_n$ is palindromic.
In particular, the constant term $W_n(0)$ equals the leading coefficient, which is positive.
Using this, one may show by induction that $W_n(z)$ has $n$ negative simple zeros and that $W_{n-1}$ strictly interlaces $W_n$. 
The proof using the method of interlacing zeros is quite classical 
and we left it for the explicit polynomial sequence $\{W_n\}$ to interested readers.
Hence every zero of $H_n$ is negative.
This completes the proof.
\end{proof}

Furthermore, we are able to show that the zeros of polynomials $F_n(z)$
are densely distributed on the half circle, see~\cref{prop:limit}.
Consider the set of zeros of all polynomials in the sequence $\mathcal W=\{W_n(z)\}_{n\ge 0}$ 
as a subset of $\mathbb R^2$ with the Euclidean topology.
A complex number $z^*$ is a \emph{limit of zeros} of the polynomials in~$\mathcal W$ 
if there is a zero $z_n$ of the polynomial $W_n(z)$ for each $n$ such that 
\[
\lim_{n\to\infty}z_n=z^*.
\]
A limit of zero is either isolated or non-isolated.
It is clear that both the set of isolated limits of zeros and
the set of non-isolated limits of zeros are closed.

In order to characterize the zero distribution of the polynomials $F_n(z)$,
we need a particular case of Wang and Zhang's result \cite[Theorem 2.3]{WZ21}.

\begin{theorem}[Wang and Zhang]\label{thm:lz}
Let $a,b,c,d\in\mathbb R$ and $ac\ne0$.
Let $W_n(z)$ be a polynomial defined by
$W_0(z)=1$, $W_1(z)=z$, and
\[
W_n(z) = (az+b)W_{n-1}(z)+(cz+d)W_{n-2}(z),
\quad\text{for $n\ge 2$.}
\]
Suppose that $\mathcal W=\{W_n(z)\}_{n\ge 0}$ is not
the sequence $\{z^n\}_{n=0}^{\infty}$ of geometric series.
Let
\[
\Delta_\Delta=c^2+a(bc-ad).
\]
If $\Delta_\Delta<0$,
then the set of isolated limits of zeros of $\mathcal W$ is
\begin{align*}
\brk[c]2{z\colon (1-a)z^2-(b+c)z-d=0,\,
\Re\brk2{(az+b)\overline{(2-a)z-b}}<0},
\end{align*}
and the set of non-isolated limits of zeros of $\mathcal W$
is the circular arc connecting the points $x_\Delta^-$ and~$x_\Delta^+$, 
through the point $-b/a$, where
\[
x_\Delta^\pm
=\frac{-ab-2c\pm2\sqrt{\Delta_\Delta}}{a^2}.
\]
\end{theorem}

Deriving the limits of zeros of the sequence $\{F_n(z)\}_{n\ge 1}$ is a direct application of \cref{thm:lz}.
\begin{proposition}\label{prop:limit}
The sequence $\{F_n(z)\}_{n\ge 1}$
has no isolated limit zeros,
and its set of non-isolated limit zeros is 
the left half part of the circle 
\[
\abs{z-1}=2,
\]
including the northern and southern poles $1\pm 2i$.
\end{proposition}
\begin{proof}
Let $W_n(z)=F_{n+1}(z-1)$. Substituting $z$ by $z-1$ in \cref{gf:A}, we obtain
\[
W_n(z)=zW_{n-1}(z)+(2-z)W_{n-2}(z)\quad\text{for $n\ge 2$},
\]
with initial values $W_0(z)=1$ and $W_1(z)=z$.
By routine computation, we obtain from \cref{thm:lz} that 
the sequence $\{W_n(z)\}_{n\ge0}$ has no isolated limits of zeros,
and its non-isolated zeros is the left half part of the circle 
\[
\abs{z-2}=2.
\]
Since $F_n(z)=W_{n-1}(z+1)$, translating the circle to the left by one unit, 
one obtains the desired results for limits of zeros of the sequence $\{F_n(z)\}_{n\ge 1}$.
\end{proof}

Combining \cref{prop:zeros,prop:limit}, we obtain a proof of \cref{thm:rg}.

\end{document}